\documentclass[12pt,a4paper]{amsart}
\usepackage{amsfonts}
\usepackage{amsthm}
\usepackage{amsmath}
\usepackage{amscd}
\usepackage[latin2]{inputenc}
\usepackage{t1enc}
\usepackage[mathscr]{eucal}
\usepackage{indentfirst}
\usepackage{graphicx}
\usepackage{graphics}
\usepackage{pict2e}
\usepackage{epic}
\numberwithin{equation}{section}
\usepackage[margin=2.9cm]{geometry}
\usepackage[all]{xy}
\usepackage{amssymb}
\usepackage{mathtools}
\usepackage[toc]{appendix}
\usepackage{MnSymbol}
\usepackage{stackrel}
\usepackage[new]{old-arrows}
\usepackage[backend=biber, style=alphabetic, sorting=ynt
]{biblatex}

\usepackage[dvipsnames]{xcolor} 
    \definecolor{red_cite}{RGB}{75, 128, 82}

\usepackage[
 colorlinks=true,        
 allcolors = black,  
 citecolor=red_cite,        
     ]{hyperref} 
\usepackage[hyphenbreaks]{breakurl}

\hypersetup{
    colorlinks=true,
    linkcolor=blue,
    filecolor=magenta,      
    urlcolor=blue,
    breaklinks=true,
} 

\theoremstyle{plain}
\newtheorem{Th}{Theorem}[section]
\newtheorem{Lemma}[Th]{Lemma}
\newtheorem{Cor}[Th]{Corollary}
\newtheorem{Prop}[Th]{Proposition}

\theoremstyle{definition}
\newtheorem{Def}[Th]{Definition}

\newtheorem{Rem}[Th]{Remark}
\newtheorem{?}[Th]{Problem}

\newcommand{\Hom}{\mathrm{Hom}}

\newcommand{\Mod}[1]{#1\text{-}\mathrm{Mod}}
\newcommand{\FdMod}[1]{#1\text{-}\mathrm{FdMod}}
\newcommand{\Cat}[1]{\mathrm{#1}}
\newcommand{\Sub}{\mathrm{Sub}}
\newcommand{\Image}{\mathrm{Im}}
\newcommand{\colim}{\mathrm{colim}}
\newcommand{\Ind}[1]{\mathrm{Ind}({#1})}
\newcommand{\Succ}[1]{\mathrm{Succ}({#1})}
\newcommand{\Full}[1]{\mathrm{Full}(#1)}
\newcommand{\mv}[2]{\mathrm{mv}(^{#1}{#2})}

\newcommand{\stackstag}[1]{\href{https://stacks.math.columbia.edu/tag/#1}{#1}}

\title{Towards constructivising the Freyd--Mitchell embedding theorem}
\author{Anna Giulia Montaruli}

\begin{document}

\begin{abstract}
The aim of the paper is to first point out that the classical proof of the Freyd--Mitchell Embedding Theorem does not work in $\mathbf{CZF}$; then, to propose an alternative embedding of a small abelian category into the category of sheaves of modules over a ringed space, which works constructively. It is necessary to mention that this work has been initially inspired by Erik Palmgren, who unexpectedly passed away in November 2019: I'm very grateful to him for having shared with me his intuitions, and for having supervised the realization of the first half of the paper.
\end{abstract}

\maketitle

\section{Introduction}
What is a good ``concrete'' description of a small abelian category? Classically, the answer to this question is given by the Freyd--Mitchell Embedding Theorem, which asserts that every small abelian category $\mathcal{A}$ admits a full, exact embedding into the category $\Mod{R}$ of modules over an appropriate ring $R$ (see \cite{Mitchell} or \cite{Freyd}). Thanks to this result, we can think about the objects of $\mathcal{A}$ as modules over some ring $R$, and of the maps in $\mathcal{A}$ as modules homomorphisms. If we analyze some different proofs, however, we come across constructive issues.
In fact, a non constructive argument has two shortcomings: on the one hand, the general project of constructivising mathematics is itself of intrinsic interest; on the other hand, non constructive arguments can often give proofs that are somehow ``mysterious'', while constructive arguments clarify more concretely what exactly is going on. The motivation of the work presented here lies in this context; indeed, its aim is to first show and explain what in the proofs of the Freyd--Mitchell Embedding Theorem fails when working within $\mathbf{CZF}$ (a constructive set theory), and then to propose an alternative constructive embedding of $\mathcal{A}$ into the category of sheaves of modules over a ringed space. 

The organization of the paper is as follows. In Section \ref{FM} we analyze the outline of a proof of the Freyd--Mitchell Embedding Theorem, and we show that, while working in $\mathbf{CZF}$, a particular small abelian category $\mathbb{G}$ gives Brouwerian counterexamples to two of their ingredients, namely the well-poweredness of the category $[\mathcal{A} , \Cat{Ab}]$ of additive functors from $\mathcal{A}$ to the category $\Cat{Ab}$ of abelian groups, and the standard construction of enough injectives for  $[\mathcal{A} , \Cat{Ab}]$. In this section, we include also an intermediate result about the existence of enough injectives in the category $\Cat{Ab}$, and we mention that some unprovability results also apply to other systems. 

In Section \ref{Description-new-embedding} we describe the constructive embedding of $\mathcal{A}$ into the category of sheaves of modules over a ringed space; this is obtained using two already known embeddings, one of which requires that the site $(\mathcal{A} , R)$, where $R$ is the regular topology, has a set of conservative points. Section \ref{site_has_enough_points} is devoted to formulating a constructive proof that this holds.

In Appendix \ref{appendix A}, the reader will find a brief description of the axioms of $\mathbf{CZF}$. We refer to \cite{CZF} for a detailed exposition. $\mathbf{CZF}$ is a fairly weak theory, so all results proven there must hold also in stronger systems such as $\mathbf{IZF}$, and those not involving unbounded quantification  work in $\mathbf{IHOL}$ too. See \cite{Scedrov} for details on $\mathbf{IZF}$, and \cite[Chapter II.1]{lambek-scott:intro} for details on $\mathbf{IHOL}$.

Throughout all the paper, whenever we reference a result from classical literature in a constructive proof, this means that the proof of the result referenced also works in $\mathbf{CZF}$. Moreover, the reader will find definitions that are classically well known, since in a constructive context one must clarify which definition is using among classically equivalent alternatives.

\subsection*{Acknowledgements}
After Erik Palmgren's unexpected death in November 2019, the supervision of this work has been continued by Peter LeFanu Lumsdaine, who has always been available with his precious comments and suggestions, and to whom I'm very grateful. Many thanks also to Johan Lindberg for some interesting discussions.

\section{Obstacles to constructivising the existing proofs}\label{FM}

The Freyd--Mitchell Embedding Theorem asserts that every small abelian category $\mathcal{A}$ admits a full, exact embedding into the category $\Mod{R}$ of modules over an appropriate ring $R$. 
The proof, as it is presented in \cite[Theorem VI.7.2]{Mitchell}, follows this outline:
\begin{enumerate}
\item[$\diamond$] the functor $\mathcal{A} \rightarrow \mathcal{L}(\mathcal{A}, \Cat{Ab})$, sending an object $A$ to $H^A \coloneqq Hom_{\mathcal{A}}(A, -)$, is a contravariant, full, exact embedding from $\mathcal{A}$ into the abelian category $\mathcal{L}(\mathcal{A}, \Cat{Ab})$ of left exact additive functors from $\mathcal{A}$ to $\Cat{Ab}$;
\item[$\diamond$] by composing this embedding with the duality functor $(-)^{op}$ 
on $\mathcal{L}(\mathcal{A}, \Cat{Ab})$, we obtain a full, exact, covariant embedding $S:\mathcal{A}\rightarrow \mathcal{L}(\mathcal{A}, \Cat{Ab})^{op}$;
\item[$\diamond$] by \cite[Theorem VI.6.2]{Mitchell}, $\mathcal{L}(\mathcal{A}, \Cat{Ab})$ has an injective cogenerator, and so $\mathcal{L}(\mathcal{A}, \Cat{Ab})^{op}$ has a projective generator;
\item[$\diamond$] every object in the image of $S$ is finitely generated with respect to the family $\{P' \coloneqq \underset{\substack{A\in\mathcal{A} \\ f : P \rightarrow H^{A^{op}}}}{\coprod} P \}$ (finitely generated in the sense of \cite[pag.72]{Mitchell});
\item[$\diamond$] by \cite[Theorem IV.4.1]{Mitchell}, defining $R \coloneqq Hom(P' , P')$, we see that the functor $T : \mathcal{L}(\mathcal{A}, \Cat{Ab})^{op} \rightarrow \Mod{R}$ sending $X$ to $Hom_{\mathcal{L}(\mathcal{A}, \Cat{Ab})^{op}}(P', X)$ is a full, exact embedding;
\item[$\diamond$] the composition $T\cdot S : \mathcal{A} \rightarrow \Mod{R}$ is the full exact embedding we are looking for.
\end{enumerate}

Some parts of this proof present non constructive features.
Indeed, to prove that $\mathcal{L}(\mathcal{A}, \Cat{Ab})$ is abelian, \cite[Paragraph VI.6]{Mitchell} shows that this category is the category of pure objects with respect to the category of monofunctors (i.e.\ functors preserving monomorphisms); to achieve this, $\mathcal{A}$ is supposed to be well-powered (i.e.\ that any subobject collection of $\mathcal{A}$ is a set), and the category $[ \mathcal{A},\Cat{Ab}]$ of additive functors from $\mathcal{A}$ to $\Cat{Ab}$ is supposed to have injective envelopes.
Moreover, to conclude that the category $\mathcal{L}( \mathcal{A},\Cat{Ab})$ has an injective cogenerator, the proof relies on the fact that the category $[ \mathcal{A},\Cat{Ab}]$ has enough injectives (weaker than the existence of injective envelopes, but still suspicious) and is well-powered.
Also, we remark that the proof that every object in the image of $S$ is finitely generated with respect to $\{P'\}$ is based on the assumption that every $Hom_{\mathcal{L}(\mathcal{A}, \Cat{Ab})^{op}}(P, (H^A)^{op})$ is detachable from $\underset{A\in\mathcal{A}}{\coprod} Hom_{\mathcal{L}(\mathcal{A}, \Cat{Ab})^{op}}(P, (H^A)^{op})$ or, equivalently, that the set of objects of $\mathcal{A}$ has decidable equality. 
There are also other slightly different proofs of the theorem; nevertheless, all of them seem to have similar issues.

In this section we show that, working in Constructive Zermelo--Fraenkel Set Theory ($\mathbf{CZF}$), a certain small abelian subcategory $\mathbb{G}$ of $\Cat{Ab}$ gives us Brouwerian counterexamples to the well-poweredness of $[\mathcal{A}, \Cat{Ab}]$ and to the standard way of constructing enough injectives in $[\mathcal{A}, \Cat{Ab}]$, i.e.\ it shows that each of these implies a classicality/impredicativity principle which is known to be unprovable in $\mathbf{CZF}$. As we mention, the unprovability result about the standard construction of enough injectives apply also to $\mathbf{IZF}$ and $\mathbf{IHOL}$.

\subsection{The category \texorpdfstring{$\mathbb{G}$}{G} of finite direct powers of \texorpdfstring{$\mathbb{Z}_2$}{Z2}}\label{G}
We define $\mathbb{G}$ as the full subcategory of $\Cat{Ab}$ whose collection of objects is given by $\{ \mathbb{Z}_2^{\oplus n} \}_{n\in \mathbb{N}}$.
We identify the object $\mathbb{Z}_2^{\oplus 0} $ with the zero object $\{ 0 \} $.

A map $\mathbb{Z}_2^{\oplus n} \stackrel{f}{\longrightarrow} \mathbb{Z}_2^{\oplus m}$ is specified by a matrix $[ f_{i,j}]_{j\in m}^{i \in n}$, whose components are given by the composition:
\[ \xymatrix{ \mathbb{Z}_2^{\oplus n} \ar[rr]^-{f} & & \mathbb{Z}_2^{\oplus m}\ar[d]^-{\pi_j} \\
\mathbb{Z}_2\ar[u]^-{m_i} \ar[rr]_-{f_{i,j}} & & \mathbb{Z}_2 } \]
where the arrows $\{m_i \}_{i\in n}$ (resp.\ $\{\pi_j \}_{j\in m}$) are the canonical injections (resp.\ projections) of the biproduct $\mathbb{Z}_2^{\oplus n}$ (resp.\ $\mathbb{Z}_2^{\oplus m}$).

Note that, since every $f_{i,j}$ is a group homomorphism from $\mathbb{Z}_2 $ to $\mathbb{Z}_2 $, then it can only be either equal to $id_{\mathbb{Z}_2}$ or to the zero map. Hence equalities of maps from $\mathbb{Z}_2 $ to $\mathbb{Z}_2 $ are decidable, meaning that, for any two maps $f$ and $g$, either $f=g$ or $f\neq g$. This decidability can be extended (componentwise) to a decidability of equality of arrows in $\mathbb{G}$.

It is easy to see that the category $\mathbb{G}$ defined above is small. Furthermore, being a full subcategory of $\Cat{Ab}$, $\mathbb{G}$ will have kernels (resp.\ cokernels) exactly when, for every map $f$ in $\mathbb{G}$, a kernel (resp.\ a cokernel) of $f$ in $\Cat{Ab}$ will lie in $\mathbb{G}$. 

\begin{Lemma}\label{kernels}
The category $\mathbb{G}$ has kernels.
\end{Lemma}
\begin{proof}
In order to prove that the category $\mathbb{G}$ has kernels, it is enough to show that any subgroup $K$ of $\mathbb{Z}_2^{\oplus n}$ is isomorphic to $\mathbb{Z}_2^{\oplus p}$ for some $p \leq n$. We already know that this result is valid in a classical setting; here, we explicitly rewrite the proof, in order to show its constructiveness.  

First, note that, by Lagrange's Theorem (whose proof still works constructively), and due to the fact that the cardinality of $\mathbb{Z}_2^{\oplus n}$ is $2^{n}$, $K$ must have cardinality equal to $2^{p}$ for some $p\leq n$.
We also know that its non zero elements (which are also elements of $\mathbb{Z}_2^{\oplus n}$) have order $2$. 
These two facts allow us to conclude that $K \cong \mathbb{Z}_2^{\oplus p} $. Indeed, one can show that, if $S$ is a maximal set of elements which are linearly independent in $K$, then $\lvert S \rvert = p$. Moreover, given any such a set $\{x_1 , \dots , x_p \}$, $K \cong \underset{i=1 , \dots , p}{\oplus} \langle x_i \rangle \cong \mathbb{Z}_2^{\oplus p}$.
\end{proof}

\begin{Lemma}\label{cokernels}
The category $\mathbb{G}$ has cokernels.
\end{Lemma}
\begin{proof}
In order to prove that the category $\mathbb{G}$ has cokernels, it is enough to show that, for any given map $f : \mathbb{Z}_2 ^{\oplus n} \rightarrow \mathbb{Z}_2 ^{\oplus m}$, a cokernel of this map in $\Cat{Ab}$ is isomorphic to an object of $\mathbb{G}$.
Consider the diagram 
\[\xymatrix{ \mathbb{Z}_2 ^{\oplus n}\ar[rr]^f\ar@{->>}[dr]^e & & \mathbb{Z}_2 ^{\oplus m} \ar@{->>}[rr]^q & & \mathbb{Z}_2 ^{\oplus m} / \Image f  \\
& \Image f \ar[ur]^j & & & }\]
where $q$ is a cokernel of $f$ in $\Cat{Ab}$.
From \ref{kernels} it follows that $\Image f $, being a subgroup of $\mathbb{Z}_2 ^ {\oplus m}$, is isomorphic to $\mathbb{Z}_2 ^{\oplus p}$ for some $0 \leq p \leq m$, and its cardinality is $2^p$. Hence, $\mathbb{Z}_2 ^{\oplus m} / \Image f $ has cardinality $2^{m-p}$. Moreover, every non zero element of $\mathbb{Z}_2 ^{\oplus m} / \Image f $ has order $2$. As in Lemma \ref{kernels}, we can conclude that $ \mathbb{Z}_2 ^{\oplus m} / \Image f  \cong \mathbb{Z}_2^{\oplus m - p}$.
\end{proof}

\begin{Prop}
The full subcategory $\mathbb{G}$ of $\Cat{Ab}$, whose objects are the finite direct powers of $\mathbb{Z}_2$, is an abelian category.
\end{Prop}
\begin{proof}
$\mathbb{G}$ has zero object $\mathbb{Z}_2^{\oplus 0}$, kernels (see Lemma \ref{kernels}) and cokernels (see Lemma \ref{cokernels});
it also inherits biproducts from $\Cat{Ab}$, by identifying $(\mathbb{Z}_2^{\oplus n}) \oplus (\mathbb{Z}_2^{\oplus m})$ with $\mathbb{Z}_2^{\oplus n+m}$.
In order to be able to conclude that the category $\mathbb{G}$ is abelian, we need to verify that every monomorphism is a kernel and that every epimorphism is a cokernel; these properties are inherited from the category $\Cat{Ab}$, using the fullness of $\mathbb{G}$. 
Indeed, given a monomorphism $l : \mathbb{Z}_2^{\oplus n} \rightarrowtail \mathbb{Z}_2^{\oplus m} $ in $\mathbb{G}$, we can show that this is the kernel of its cokernel. From Lemma \ref{cokernels} we know that the cokernel of $l$ in $\mathbb{G}$ is some arrow $cokerl(l) : \mathbb{Z}_2^{\oplus m} \twoheadrightarrow \mathbb{Z}_2^{\oplus m-n}$.
Clearly, $coker(l)\cdot l = 0$; it remains to show that $l$ has the universal property: but this is inherited from the one in $\Cat{Ab}$.
In the same way, it is possible to show that every epimorphism is the cokernel of its kernel.
\end{proof}

\begin{Rem}
Note that the category $\mathbb{G}$ is strongly equivalent to the category\break $\FdMod{\mathbb{Z}_2}$ of finite-dimensional $\mathbb{Z}_2$-modules equipped with a chosen basis. Indeed, the functor $F: \mathbb{G} \rightarrow \FdMod{\mathbb{Z}_2}$, which sends $\mathbb{Z}_2^{\oplus n}$ to $\mathbb{Z}_2^{\oplus n}$ equipped with the obvious basis, and the functor $G : \FdMod{\mathbb{Z}_2} \rightarrow \mathbb{G}$, which sends $M$ equipped with a base $\mathcal{B}$ to $\mathbb{Z}_2^{\oplus \mid \mathcal{B} \mid}$, are quasi-inverses. 
\end{Rem}

\subsection{Constructive issues}

The aim of this subsection is to witness the constructive issues previously mentioned: the well-poweredness of $[\mathcal{A}, \Cat{Ab}]$ and the existence of enough injectives. Throughout all the subsection, $\mathbb{G}$ denotes the small abelian category of finite direct powers of $\mathbb{Z}_2$ introduced in Subsection \ref{G}.

In $\mathbf{CZF}$, the Axiom of Power Set implies the well-poweredness of $[\mathbb{G}, \Cat{Ab}]$. In fact, they are equivalent: we prove here that, in $\mathbf{CZF}$, the well-poweredness of $[\mathbb{G}, \Cat{Ab}]$ implies the Axiom of Power Set. 

\begin{Lemma}\label{subgroups}
For every $F, S\in [\mathbb{G}, \Cat{Ab}]$, $S$ is a subfunctor of $F$ if and only if $S(\mathbb{Z}_2)$ is isomorphic to a subgroup of $F(\mathbb{Z}_2)$.
\end{Lemma}
\begin{proof}
If $S$ is a subfunctor of $F$, then there exists a monomorphism $S(\mathbb{Z}_2) \stackrel{f}{\rightarrowtail} F(\mathbb{Z}_2)$. Hence $S(\mathbb{Z}_2) \cong \Image f \leq F(\mathbb{Z}_2)$.

To prove the other implication, note that any functor $F\in [\mathbb{G}, \Cat{Ab}]$ is uniquely determined on objects by its behaviour on $\mathbb{Z}_2 $.
Furthermore, looking at the diagram
\[ \xymatrix{  F(\mathbb{Z}_2^{\oplus n})\ar[rr]^{F(f)}\ar[d]^{\wr}_{\phi_n} & & F( \mathbb{Z}_2^{\oplus m} ) \ar[d]^{\wr}_{\phi_m} \\
F(\mathbb{Z}_2)^{\oplus n} \ar[rr]^{\phi_m\cdot F(f) \cdot \phi_n^{-1}} & & F(\mathbb{Z}_2)^{\oplus m}\ar@{->>}[d]^{\pi_j} \\
F(\mathbb{Z}_2) \ar[u]^{m_i}\ar[rr]^{F(f_{i,j})} & & F(\mathbb{Z}_2)}
\]
where $\phi_n$ and $\phi_m$ are the canonical isomorphisms, it is easy to see that $\phi_m\cdot F(f) \cdot \phi_n^{-1}$ (and hence also $F(f)$) is uniquely determined by the maps $F(f_{i,j}): F(\mathbb{Z}_2 ) \rightarrow F(\mathbb{Z}_2)$. Since $F$ is additive,
\[ f_{i,j} = id_{\mathbb{Z}_2} \Rightarrow F(f_{i,j}) = id_{F(\mathbb{Z}_2)} ; \ \ \ f_{i,j} = 0 \Rightarrow F(f_{i,j}) = 0 \]
Hence $F(f_{i,j}) = id_{F(\mathbb{Z}_2)}$ or $F(f_{i,j})=0$, and the behaviour of $F$ on $f$ is uniquely determined by $f$ itself.
It follows that, to define a functor in $[\mathbb{G}, \Cat{Ab}]$, it is enough to declare its image at $\mathbb{Z}_2$ and, to define a natural transformation between two functors of $[\mathbb{G}, \Cat{Ab}]$, it is enough to define its behaviour at $\mathbb{Z}_2$.

Now, to complete the proof note that, given any group $\overline{S}$ which admits a monomorphism $\overline{f}$ (not necessarily an inclusion) into $F(\mathbb{Z}_2)$, we can define the functor $S\in [\mathbb{G}, \Cat{Ab}]$ as $S(\mathbb{Z}_2) \coloneqq \overline{S}$, and the monomorphism $f_{.}: S \rightarrowtail F$ as the natural transformation which gives $\overline{f}$ when evaluated at $\mathbb{Z}_2$. 
\end{proof}

Using Lemma \ref{subgroups}, we can prove:

\begin{Prop}\label{well-poweredness}
If the category $[\mathbb{G},\Cat{Ab}]$ is well-powered, then the Axiom of Power Set holds.
\end{Prop}
\begin{proof}
Consider the functor $F$ which acts as the inclusion of $\mathbb{G}$ into $\Cat{Ab}$. From the discussion above, we see that the collection $\Sub(F)$ of all subobjects (in this case isomorphic classes of subfunctors) of $F$ can be identified with the collection of all subgroups of $\mathbb{Z}_2$ in $\Cat{Ab}$.
Hence, if $\Sub(F)$ is a set, then the collection \{subgroups  of \ $\mathbb{Z}_2$\} is a set too.
This clearly implies that the power set of the singleton is a set, which is equivalent (assuming the Axiom of Exponentiation) to have the full Axiom of Power Set.
\end{proof}

Injectivity and the existence of enough injectives are defined constructively just like classically:
\begin{Def}
An object $I$ of a category $\mathcal{C}$ is called injective if, given any monomorphism $f: A \rightarrowtail B$ of $\mathcal{C}$ and any map $h : A \rightarrow I$, there exists a map $g : B \rightarrow I$ such that $g\cdot f = h$.

We say that $\mathcal{C}$ has enough injectives if, for every object $A$ of $\mathcal{C}$, there exists a monomorphism from $A$ into an injective object of $\mathcal{C}$.
\end{Def}

In the following, we are considering injectives which are decidable-valued functors.
\begin{Def}
A functor whose target category is concrete is said to be a decidable-valued functor if it is valued in objects whose underlying sets have decidable equalities.
\end{Def}

We are now going to show that, in $[\mathbb{G}, \Cat{Ab}]$, the existence of enough decidable-valued injectives implies the Weak Law of Excluded Middle (WLEM) for restricted formulas. Analogous results for sets have been given in \cite{Aczel-etc}.

\begin{Prop}\label{enough-injectives}
If the category $[\mathbb{G}, \Cat{Ab}]$ has enough decidable-valued injectives, then WLEM holds.
\end{Prop}
\begin{proof}
Fix a formula $\phi$ with only bounded quantifiers, and define the set
\begin{equation} \label{p}
p \coloneqq \{ x\in\mathbb{N} \mid (x=0)\wedge \phi \}
\end{equation}

By identifying $\bot$ with the empty set and $\top$ with the singleton $\{ 0 \}$, we clarly have 
\[\neg \phi \Leftrightarrow p= \bot \ \ \text{and} \ \ \neg\neg \phi \Leftrightarrow p \neq \bot \]
and so 
\begin{equation} \label{p-wlem}
\neg \phi \vee \neg \neg \phi \Leftrightarrow p= \bot \vee p\neq \bot
\end{equation} 
For a given set $J$, let $A^{\oplus J}$ be the external direct sum defined as in \cite{Richman}.

Consider the monomorphism in $[\mathbb{G},\Cat{Ab}]$
\[ l_{.}: \Hom(\mathbb{Z}_2 , - )^{\oplus \{ \bot , \top \} } \rightarrowtail \Hom(\mathbb{Z}_2 , - )^{\oplus \{ \bot , \top, p \} } \]
whose component at $\mathbb{Z}_2^{\oplus n}$ acts as 
\[ l_{\mathbb{Z}_2^{\oplus n}} : \Hom(\mathbb{Z}_2 , \mathbb{Z}_2^{\oplus n} )^{\oplus \{ \bot , \top \} } \rightarrowtail \Hom(\mathbb{Z}_2 , \mathbb{Z}_2^{\oplus n} )^{\oplus \{ \bot , \top, p \} } \]
\[ (\alpha_{\bot}, \alpha_{\top}) \longmapsto (\alpha_{\bot}, \alpha_{\top}, 0_p) \]
where, in the indices, we keep track of the generator involved (e.g.\ $\alpha_{\bot}$ indicates that we take the map $\alpha$ in the hom-set corresponding to the generator $\bot$). 
Note that, if $p=\bot$ or if $p=\top$, then $l_{.}$ turns out to be the identity map.

Assuming that the category $[\mathbb{G},\Cat{Ab}]$ has enough injectives, then there exists a monomorphism $f_{.}: \Hom(\mathbb{Z}_2 , - )^{\oplus \{ \bot , \top \} } \rightarrowtail I$, where $I$ is an injective object of $[\mathbb{G}, \Cat{Ab}]$. 

This implies that we have a commutative diagram of the shape

\[ \xymatrix{ I & &\\
\Hom(\mathbb{Z}_2 , - )^{\oplus \{ \bot , \top \} } \ar[u]^{f_{.}}\ar[rr]_-{l_{.}} & & \Hom(\mathbb{Z}_2 , - )^{\oplus \{ \bot , \top, p \} } \ar[ull]_-{g_{.}} } \]

that, evaluated at $\mathbb{Z}_2$, gives the commutative diagram of abelian groups:

\[ \xymatrix{ I(\mathbb{Z}_2) & &\\
\Hom(\mathbb{Z}_2 , \mathbb{Z}_2 )^{\oplus \{ \bot , \top \} } \ar[u]^{f_{\mathbb{Z}_2}}\ar[rr]_-{l_{\mathbb{Z}_2}} & & \Hom(\mathbb{Z}_2 , \mathbb{Z}_2 )^{\oplus \{ \bot , \top, p \} } \ar[ull]_-{g_{\mathbb{Z}_2}}}\]

Call $y \coloneqq g_{\mathbb{Z}_2} (0_{\bot},0_{\top},id_p))$, and consider $p$. If $p=0$, then $l_{\mathbb{Z}_2}$ is the identity, and $y = f_{\mathbb{Z}_2}( id_{\bot},0_{\top})$. Suppose now that $y=f_{\mathbb{Z}_2}( id_{\bot},0_{\top})$; if $p=\top$, using the fact that $f_{.}$ is pointwise a monomorphism, $y = f_{\mathbb{Z}_2} (0_{\bot},id_{\top}) \neq f_{\mathbb{Z}_2}(id_{\bot} , 0_{\top})$; hence $p$ is not inhabited, and therefore $ p = \bot$.

Thus, we have 
 \[ p = \bot \Leftrightarrow y=f_{\mathbb{Z}_2}( id_{\bot}, 0_{\top}) \]
from which
\[ \neg \phi \vee \neg \neg \phi \Leftrightarrow p= \bot \vee p\neq \bot \Leftrightarrow y=f_{\mathbb{Z}_2}( id_{\bot}, 0_{\top}) \vee y\neq f_{\mathbb{Z}_2}( id_{\bot}, 0_{\top}) \]

Then, as long as $y=f_{\mathbb{Z}_2}( id_{\bot}, 0_{\top}) \vee y\neq f_{\mathbb{Z}_2}( id_{\bot}, 0_{\top})$ is true, we are done. This is the case, for instance, if the abelian group $I(\mathbb{Z}_2)$ has decidable equality.
\end{proof}

The discussion of this subsection can be summarized with the following results.
\begin{Th} \leavevmode
\begin{enumerate}
    \item If, for all small abelian categories $\mathcal{A}$, $[\mathcal{A}, \Cat{Ab}]$ is well-powered, then the Axiom of Power Set holds.
    \item If, for all small abelian categories $\mathcal{A}$, $[\mathcal{A}, \Cat{Ab}]$ has enough decidable-valued injectives, then WLEM holds.
\end{enumerate}
\end{Th}
\begin{proof}
By Propositions \ref{well-poweredness} and \ref{enough-injectives}, the category $\mathbb{G}$ of finite direct powers of $\mathbb{Z}_2$ shows both these implications. 
\end{proof}

\begin{Cor}
\leavevmode
\begin{enumerate}
    \item The statement ``for every small abelian category $\mathcal{A}$, $[\mathcal{A}, \Cat{Ab}]$ is well-powered'' is not provable in $\mathbf{CZF}$.
    \item The statement ``for every small abelian category $\mathcal{A}$, $[\mathcal{A}, \Cat{Ab}]$ has enough de\-cidable-valued injectives'' is not provable in $\mathbf{IZF}$, $\mathbf{IHOL}$, $\mathbf{CZF}$.
\end{enumerate}
\end{Cor}
\begin{proof}
From the fact that the Axiom of Power Set is not provable in $\mathbf{CZF}$, and that WLEM is not provable in $\mathbf{IZF}$, $\mathbf{IHOL}$ and $\mathbf{CZF}$
\end{proof}

\begin{Rem}
Even if Proposition \ref{enough-injectives}, considering only decidable-valued functors, gives a weak result, it is still of some interest, since WLEM is not derivable not only in $\mathbf{CZF}$, but in every constructive system (like, for instance, the Intuitionistic Zermelo--Fraenkel Set Theory $\mathbf{IZF}$). Proposition \ref{well-poweredness} gives a full result leading to the unprovability of the well-poweredness of $[\mathbb{G}, \Cat{Ab}]$ only in $\mathbf{CZF}$, since the Axiom of Power Sets is part of the theory of constructive systems different from $\mathbf{CZF}$.
\end{Rem}

\subsection{About the existence of injective abelian groups}
As we have shown that the existence of enough decidable-valued injectives in the category $[\mathbb{G}, \rm{Ab}]$ implies WLEM for restricted formulas, one can also show that the existence of enough injectives with decidable equalities in the category $\Cat{Ab}$ implies WLEM.

\begin{Lemma}\label{enough-injectives-groups}
If there exists a monomorphism $f$ in $\Cat{Ab}$ from $\mathbb{Z}^{\{ \bot , \top \} }$ into an injective object $I$ which has decidable equality, then WLEM holds. 
\end{Lemma}

\begin{proof}
Given $p$ as in \ref{p}, consider the monomorphism
\[ l: \mathbb{Z}^{\oplus \{\bot,\top\}} \rightarrowtail \mathbb{Z}^{\oplus \{\bot,\top , p\}} \]
\[ (a_{\bot},b_{\top}) \longmapsto (a_{\bot},b_{\top},0_p) \]
then, using the injectivity of $I$, we get a commutative diagram of abelian groups
\[ \xymatrix{ I & \\
 \mathbb{Z}^{\oplus \{\bot,\top\}}\ar[u]^-{f}\ar[r]_-{l} & \mathbb{Z}^{\oplus \{\bot,\top , p\}}\ar[ul]_-{g} }\]
Note that, if $p=\bot$, then $g(0_{\bot}, 0_{\top}, 1_p) = f(1_{\bot} , 0_{\top})$. Suppose now that $g(0_{\bot}, 0_{\top}, 1_p) = f(1_{\bot} , 0_{\top})$: if $p = \top$, then $g(0_{\bot}, 0_{\top}, 1_p) = f(0_{\bot} , 1_{\top})$; but, since $f$ is injective, $f(0_{\bot} , 1_{\top}) \neq f(1_{\bot} , 0_{\top})$; therefore $p$ cannot be inhabited, and so $p = \bot$.

Hence, called $y \coloneqq g(0_{\bot}, 0_{\top}, 1_p)$,
\[ p=\bot \vee p \neq \bot \Leftrightarrow y = f(1_{\bot} , 0_{\top}) \vee y \neq f(1_{\bot} , 0_{\top}) \]

Then, as long as $y = f(1_{\bot} , 0_{\top}) \vee y \neq f(1_{\bot} , 0_{\top})$ is true, and using \ref{p-wlem}, we are done; if $I$ has decidable equality, this is the case.
\end{proof}

\begin{Cor}\label{e.i.WLEM}
The statement ``the category $\Cat{Ab}$ has enough injectives with decidable equalities'' implies WLEM.
\end{Cor}
\begin{proof}
Straight from Lemma \ref{enough-injectives-groups}.
\end{proof}

\begin{Cor}\label{cor2.12}
The statement ``either $\mathbb{Q}$ or $\mathbb{Q}/ \mathbb{Z}$ is injective in $\Cat{Ab}$'' implies WLEM.
\end{Cor}
\begin{proof}
We instantiate the map $f$ of Lemma \ref{enough-injectives-groups} with the morphism sending $(1_{\bot} , 0_{\top})$ to $\frac{1}{2}$ and $(0_{\bot} , 1_{\top})$ to $\frac{1}{3}$ (resp.\ sending $(1_{\bot} , 0_{\top})$ to $[\frac{1}{2}]_{\mathbb{Z}}$ and $(0_{\bot} , 1_{\top})$ to $[\frac{1}{3}]_{\mathbb{Z}}$, where $[x]_{\mathbb{Z}}$ denotes the projection of $x$ onto $\mathbb{Q} / \mathbb{Z}$); even if this is not a monomorphism, the rest of the proof of Lemma \ref{enough-injectives-groups} still works. Since the equality of $\mathbb{Q}$ (resp.\ $\mathbb{Q} / \mathbb{Z}$) is decidable, we get that the injectivity of $\mathbb{Q}$ (resp.\ $\mathbb{Q}/ \mathbb{Z}$) implies WLEM.
\end{proof}

In the constructive approach, divisible groups are defined as follows:
\begin{Def}
An abelian group $G$ such that, for any $x\in G$, for any $n\in\mathbb{N}$, there exists $y\in G$ such that $n y = x$, is called divisible.
\end{Def}

\begin{Cor} 
The statement ``any divisible abelian group is injective'' implies WLEM.
\end{Cor}
\begin{proof}
Direct from the proof of Corollary \ref{cor2.12}, since the classical proofs of the divisibility of $\mathbb{Q}$ and of $\mathbb{Q}/\mathbb{Z}$ work constructively. 
\end{proof}

\begin{Cor}\label{abelian-czf}
\leavevmode
\begin{enumerate}
    \item The statement ``the category of abelian groups has enough injectives with decidable equalities'' is not provable in $\mathbf{IZF}$, $\mathbf{IHOL}$, $\mathbf{CZF}$.
    \item The statement ``either $\mathbb{Q}$ or $\mathbb{Q} / \mathbb{Z}$ is injective in $\Cat{Ab}$'' is not provable in $\mathbf{IZF}$, $\mathbf{IHOL}$, $\mathbf{CZF}$.
    \item The statement ``any divisible abelian group is injective'' is not provable in $\mathbf{IZF}$, $\mathbf{IHOL}$, $\mathbf{CZF}$.
\end{enumerate}
\end{Cor}
\begin{proof}
From the fact that WLEM is not provable in $\mathbf{IZF}$, $\mathbf{IHOL}$ and $\mathbf{CZF}$.
\end{proof}

\begin{Rem}
Point (2) of \ref{abelian-czf} shows that the standard construction of enough injectives, performed using of the injectivity of $\mathbb{Q} / \mathbb{Z}$ (see Proposition \cite[I.8.3]{Hilton}), can't work constructively.
\end{Rem}

\begin{Rem}
It is already known that, in $\mathbf{ZFA}$ (i.e.\ the  Zermelo--Fraenkel set theory with atoms), the statement ``every divisible abelian group is injective'' is equivalent to the Axiom of Choice (see \cite[Theorem 2.1]{Blass}).
\end{Rem}

\begin{Rem}
One might ask if Corollary \ref{e.i.WLEM} can be strengthened, getting rid of the hypothesis on decidable equality or deriving LEM instead of WLEM. We point out that the statement ``for every ring $R$, $\Mod{R}$ has enough injectives'' holds in the stack semantics of any Grothendieck topos. Indeed, as shown in \cite[Theorem 01DU]{Stacks-project}, for every Grothendieck topos $\mathcal{E}$ and for every $R\in Ring(\mathcal{E})$, $\mathcal{E}$ has enough external injective $R$-modules, and, as stated after \cite[Theorem 3.8]{Blechschmidt}, any externally injective $R$-module is also internally injective, if $\mathcal{E}$ is supposed to have a natural number object. From this, one can derive that ``for every ring $R$, $\Mod{R}$ have enough injectives'' holds in the stack semantics of the topos $\mathcal{E}$. Hence, in the logic of a Grothendieck topos stack semantics, the statement ``for every ring $R$, $\Mod{R}$ has enough injectives'' can't imply any classicality principle. We a little extra work, this should imply that the same statement can't imply any classicality principle in $\mathbf{IZF}$. A similar remark can be found in \cite[pag.15]{Blechschmidt}
\end{Rem}

\section{Description of the new embedding}\label{Description-new-embedding}

In this section we describe how to embed a small abelian category $\mathcal{A}$ into the category of sheaves of modules over a ringed space. The embedding we construct is given by the composition of two ingredients.

First, as detailed in \cite[Appendix A]{Exact-categories}, we know the Yoneda functor gives a full exact embedding
\begin{equation*}
\mathcal{A} \rightarrow \Cat{Ab}(\mathcal{A} , R)
\end{equation*}
where $R$ is the regular topology, i.e.\ the topology whose covering families are given by single regular epimorphisms, and $\Cat{Ab}(\mathcal{A} , R)$ is the category of sheaves of abelian groups over the site $(\mathcal{A} , R)$ or, equivalently, the category of abelian group objects of the category $\Cat{Sh}(\mathcal{A} , R)$ of sheaves over the same site. 

Then, for any given topos $\mathcal{T}$ with a conservative set of points, \cite{Butz-Moerdijk} constructs a topological space $X_{\mathcal{T}}$ and a geometric morphism 
\begin{equation*}
\Cat{Sh}(X_{\mathcal{T}}) \stackrel{\Phi}{\longrightarrow} \mathcal{T}
\end{equation*}
(here $\Cat{Sh}(X_{\mathcal{T}})$ denotes the category of sheaves over $X_{\mathcal{T}}$), whose inverse image 
\begin{equation*}
\phi^* : \mathcal{T} \longrightarrow \Cat{Sh}(X_{\mathcal{T}})    
\end{equation*}
is a full and exact embedding.

Suppose that the topos $\Cat{Sh}(\mathcal{A} , R)$ has a conservative set of points. Considering, as $\mathcal{T}$, the topos $\Cat{Sh}(\mathcal{A} , R)$, and looking at the diagram

\begin{equation*}
    \xymatrix{  \mathcal{A} \ar[r]_-J & \Cat{Ab}(\mathcal{A} , R) \ar[d]\ar[rr]^-{\phi^{\ast}_{\mid \Cat{Ab}(\mathcal{A} , R)}} & & \Cat{Ab}(X_{\Cat{Sh}(\mathcal{A} , R)})\ar[d] \\
    & \Cat{Sh}(\mathcal{A} , R) \ar[rr]^-{\phi^{\ast}} & & \Cat{Sh}(X_{\Cat{Sh}(\mathcal{A} , R)})}
\end{equation*}
we can see that, since $\phi^{\ast}$ is exact, it sends abelian group objects to abelian group objects; moreover the lifting of $\phi^{*}$ to $\Cat{Ab}(\mathcal{A} , R)$ is a full, exact embedding into $\Cat{Ab}(X_{\Cat{Sh}(\mathcal{A} , R)})$, the category of sheaves of abelian groups over the topological space $X_{\Cat{Sh}(\mathcal{A} , R)}$. 

Hence the composition 
\begin{equation} \label{embedding}
\mathcal{A} \xrightarrow{\phi^{\ast}_{\mid \Cat{Ab}(\mathcal{A} , R)} \cdot J}  \Cat{Ab}(X_{\Cat{Sh}(\mathcal{A} , R)}) 
\end{equation}
is a full, exact embedding of the abelian category $\mathcal{A}$ into $\Cat{Ab}(X_{\Cat{Sh}(\mathcal{A} , R)})$, or, equivalently, into sheaves of modules over the ringed space $(X_{\Cat{Sh}(\mathcal{A} , R)}, \mathbb{Z}_{X_{\Cat{Sh}(\mathcal{A} , R)}})$, where $\mathbb{Z}_{X_{\Cat{Sh}(\mathcal{A} , R)}}$ is the constant sheaf on $X_{\Cat{Sh}(\mathcal{A} , R)}$ with image $\mathbb{Z}$. 

To be able to state the existence of the embedding \ref{embedding}, it remains to prove that the topos $\Cat{Sh}(\mathcal{A} , R)$ has a conservative set of points. This is classically true (see \cite[Expos\'{e} VI, Appendix]{SGA4} and \cite[Corollary 2.2.12]{Elephant}); in the next section, we show that it is also possible to give a constructive proof for it.

\section{\texorpdfstring{$(\mathcal{A} , R)$}{(\textit{A}, R)} has a conservative set of points}\label{site_has_enough_points}

In all the following, we assume the reader has some familiarity with sheaf theory. We refer to \cite{MacLane-Moerdijk} and to \cite{Stacks-project} for basic definitions and results.

A point of a topos $\mathcal{T}$ is defined to be a geometric morphism $p$ from $\Cat{Set}$ to $\mathcal{T}$ (see, for instance, \cite{MacLane-Moerdijk}). If $(\mathcal{C}, I)$ is a site of definition of $\mathcal{T}$, points can be equivalently defined in terms of the site, as follows.

\begin{Def}\label{def_point}
A point for the site $(\mathcal{C} , I)$ is a functor $u : \mathcal{C} \rightarrow Set$ such that 
\begin{enumerate}
    \item for every covering family $\{ U_k \rightarrow U \}_{k\in K}$ of $(\mathcal{C} , I)$, the map $\underset{k\in K}{\coprod} u(U_k) \rightarrow u(U)$ is surjective;
    \item for every covering family $\{ U_k \rightarrow U \}_{k\in K}$ and for every morphism $V \rightarrow U$, the maps $u(U_k \times_{U} V) \rightarrow u(U_k) \times_{u(U)} u(V)$ are bijective;
    \item the stalk functor $(-)_p : Sh(\mathcal{C}) \rightarrow Set$ associated to $u$, defined by 
    \begin{equation*}
F \longmapsto F_p \coloneqq \underset{\{(U, x) \mid U\in \mathcal{C}, x \in u(U)\}^{op}}{colim} F(U)
    \end{equation*}
is left exact.
\end{enumerate}
\end{Def}

For the equivalence between the two definitions, see  \cite[Section \stackstag{00Y3}]{Stacks-project}. Note that, even if this is a classical reference, it is possible to show that the proof we refer to still works constructively.

We will often pass without comment between viewing a point as a functor $u$ on the site, and as its associated geometric morphism $p$.

\begin{Def}\label{def_enough_points}
$\{ p_k \}_{k\in K}$ family of points for the site $(\mathcal{C}, I)$ is said to be conservative if, given any $\Phi : F \rightarrow G$ map in $\Cat{Sh}(\mathcal{C}, I)$, if for all $k\in K$ $\Phi_{p_k} : F_{p_k} \rightarrow G_{p_k}$ is an isomorphism, then $\Phi$ is an isomorphism. 
\end{Def}

More than Definition \ref{def_enough_points}, in this paper we will use the following sufficient condition, which is a constructive partial reformulation of the one that can be found in
 \cite[Lemma \stackstag{00YM}]{Stacks-project}.

\begin{Lemma}\label{suff_cond}
Let $\{p_k \}_{k\in K}$ be a family of points for the site $(\mathcal{C}, I)$. Suppose that, for every $F \in \Cat{Sh}(\mathcal{C}, I)$, for every $U\in Ob(\mathcal{C})$ and for every $s, s' \in F(U)$, there exists some $k\in K$ and $x\in u_k (U)$ such that, if $(U , x, s) = (U, x, s')$ in $F_{p_k}$, then $s = s'$.
Then the family $\{p_k \}_{k\in K}$ is conservative.

\end{Lemma}
\begin{proof}
Let $\Phi : F \rightarrow G$ be a map of sheaves, and suppose that, for every $k\in K$, $\Phi_{p_k}$ is an isomorphism. We want to show that $\Phi$ is an isomorphism too. Indeed:

\begin{itemize}
 \item[$\diamond$] fixed $U\in Ob(\mathcal{C})$, and given $y, y'$ in $G(U)$, suppose that $\Phi_U (y) = \Phi_U (y')$. Then, for every $k\in K$ and for every $x\in p_k (U)$, $(U, x, \Phi_U (y)) = (U, x , \Phi_U (y'))$ in $G_{p_k}$.

Since, for every $k\in K$, $\Phi_{p_k}$ is an isomorphism, and since, for every $k\in K$ and for every $x\in p_k (U)$, $\Phi_{p_k} (U, x, y)=(U, x, \Phi_U (y))=(U, x, \Phi_U (y')) = \Phi_{p_k} (U, x, y')$, then, for every $k\in K$ and for every $x\in p_k (U)$, $(U, x , y) = (U, x, y')$. By hypothesis, we can conclude that $y = y'$. Hence $\Phi$ is a monomorphism, because it is so componentwise;
\item[$\diamond$] we can show that $G \amalg_{F} G \rightarrow G$ is an isomorphism (equivalent to say that $\Phi$ is an epimorphism). The surjectivity follows from the definition of the map, whereas the injectivity can be shown as we did for $\Phi$, since the stalk functor is exact, and so the codiagonal map is stalkwise an isomorphism. \qedhere
\end{itemize} 
\end{proof}

The rest of this section is devoted to prove the existence of a conservative set of points for the site $(\mathcal{A} , R)$, with $\mathcal{A}$ a small abelian category and $R$ the regular topology on it. 

\subsection{Construction of the pair \texorpdfstring{$(\mathbb{I}, J_{\mathbb{I}})$}{(I, J(I))}.}\label{Constructing_I}

The construction we are going to perform was inspired by \cite[Section \stackstag{00YN}]{Stacks-project}; however, it has been modified for constructive purposes.
Throughout the whole Subsection \ref{Constructing_I}, $(\mathcal{C}, R)$ will denote a small site equipped with the regular topology.

Consider a directed set $(\overline{J} , \leq)$, and suppose given a functor $G_{\mathbb{J}} : \mathbb{J} \rightarrow \mathcal{C}$, where $\mathbb{J} \coloneqq (\overline{J} , \leq)^{op}$. We define the functor $u_{\mathbb{J}} : \mathcal{C} \rightarrow Set$ as 
\begin{equation}\label{u_J}
u_{\mathbb{J}} (V) \coloneqq \underset{j\in\mathbb{J}}{\colim} \ \Hom_{\mathcal{C}} (G_{\mathbb{J}} (j) , V)    
\end{equation}
Note that the associated stalk (we allow this naming, even if $p_{\mathbb{J}}$ may or may not be a point) turns out to be
\begin{equation*}
F_{p_{\mathbb{J}}} = \underset{j\in\mathbb{J}}{\colim} \ F(G_{\mathbb{J}} (j))
\end{equation*}

\begin{Def}\label{refinement} Given two pairs $(\mathbb{J}_1 , G_{\mathbb{J}_1})$ and $(\mathbb{J}_2 , G_{\mathbb{J}_2})$ as above, we say that $(\mathbb{J}_2 , G_{\mathbb{J}_2})$ is a refinement of $(\mathbb{J}_1 , G_{\mathbb{J}_1})$ if it is equipped with a full faithful functor $\mathbb{J}_1 \stackrel{i}{\rightarrow} \mathbb{J}_2$ such that $G_{\mathbb{J}_2} \cdot i = G_{\mathbb{J}_1}$.
\end{Def}
If $(\mathbb{J}_2 , G_{\mathbb{J}_2})$ is a refinement of $(\mathbb{J}_1 , G_{\mathbb{J}_1})$, then we have two natural transformations, both given, componentwise, by the universal maps of the colimits involved:
\begin{equation*}
    (\psi_{\mathbb{J}, \mathbb{I}})_{.} : u_{\mathbb{J}} \rightarrow u_{\mathbb{I}} \ \ \ (\overline{\psi_{\mathbb{J}, \mathbb{I}}})_{.} : F_{p_{\mathbb{J}}} \rightarrow F_{p_{\mathbb{I}}}
\end{equation*}

Let $(\mathbb{J}, G_{\mathbb{J}})$ be a pair as above. We define $E_{\mathbb{J}}$ to be the collection of all the triples $(j , f, \epsilon)$ such that $j\in \mathbb{J}$, $f$ is a map from $G_{\mathbb{J}}(j)$ to some $W \in \mathcal{C}$, and $\{W' \xrightarrow{\epsilon} W\}$ is a covering family for $(\mathcal{C}, R)$ (i.e.\ $\epsilon$ is a regular epimorphism in $\mathcal{C}$). The smallness of $\mathcal{C}$ ensures us that $E_{\mathbb{J}}$ is a set.

\begin{Def}\label{goodness for}
Given $(\mathbb{J}_1 , G_{\mathbb{J}_1})$ and $(\mathbb{J}_2 , G_{\mathbb{J}_2})$ as in Definition \ref{refinement}, and $e \coloneqq (j , f , W' \xrightarrow{\epsilon} W) \in E_{\mathbb{J}_1}$, we say that $(\mathbb{J}_2 , G_{\mathbb{J}_2})$ is good for $e$ if $(\psi_{\mathbb{J}_1 , \mathbb{J}_2})_W (j , f ) \in \Image (u_{\mathbb{J}_2} (\epsilon))$.

We say that $(\mathbb{J}_2 , G_{\mathbb{J}_2})$ is good for $E_{\mathbb{J}_1}$ if it is good for all its elements.
\end{Def}

Definition \ref{goodness for} and the following construction are all motivated by the need to find a refinement $(\mathbb{I}, G_{\mathbb{I}})$ of a given $(\mathbb{J}, G_{\mathbb{J}})$ which is good for the whole set $E_{\mathbb{I}}$: this condition will ensure us that the associated functor $u_{\mathbb{I}}$ defines a point $p_{\mathbb{I}}$. 

\begin{Lemma}\label{Lemma_refinements}
Let $(\mathbb{J}_2 , G_{\mathbb{J}_2})$ be a refinement of $(\mathbb{J}_1 , G_{\mathbb{J}_1})$, and let $e$ be an element of $ E_{\mathbb{J}_1}$.  If $(\mathbb{J}_2 , G_{\mathbb{J}_2})$ is good for $e$, then every refinement of $(\mathbb{J}_2 , G_{\mathbb{J}_2})$ is so.
\end{Lemma}
\begin{proof}
Suppose $e = (j, f , \epsilon)$, and let $W$ be the codomain of $\epsilon$ and of $f$; by hypothesis, $(\psi_{\mathbb{J}_1 , \mathbb{J}_2})_W (j ,f)\in \Image (u_{\mathbb{J}_2} (\epsilon))$. Given a refinement $(\mathbb{J}_3 , G_{\mathbb{J}_3})$ of $(\mathbb{J}_2 , G_{\mathbb{J}_2})$, from the commutative diagram
\begin{equation*}
\xymatrix{ u_{\mathbb{J}_1} (W') \ar[d]_{u_{\mathbb{J}_1}(\epsilon)} \ar[r]^{(\psi_{\mathbb{J}_1, \mathbb{J}_2})_{W'}} & u_{\mathbb{J}_2} (W') \ar[d]_{u_{\mathbb{J}_2}(\epsilon)} \ar[r]^{(\psi_{\mathbb{J}_2, \mathbb{J}_3})_{W'}} & u_{\mathbb{J}_3} (W') \ar[d]^{u_{\mathbb{J}_3}(\epsilon)} \\
u_{\mathbb{J}_1}(W') \ar[r]_{(\psi_{\mathbb{J}_1 , \mathbb{J}_2})_W } & u_{\mathbb{J}_2}(W') \ar[r]_{(\psi_{\mathbb{J}_2 , \mathbb{J}_3})_W} & u_{\mathbb{J}_3}(W)}
\end{equation*}
and from the fact that $(\psi_{\mathbb{J}_1 , \mathbb{J}_3})_W = (\psi_{\mathbb{J}_2 , \mathbb{J}_3})_W \cdot (\psi_{\mathbb{J}_1 , \mathbb{J}_2})_W$, we see that 
\[ (\psi_{\mathbb{J}_1 , \mathbb{J}_3})_W (j , f) \in \Image ((\psi_{\mathbb{J}_2 , \mathbb{J}_3})_W \cdot u_{\mathbb{J}_2} (\epsilon)) = \Image ( u_{\mathbb{J}_3} (\epsilon) \cdot (\psi_{\mathbb{J}_2 , \mathbb{J}_3})_{W'}) \subseteq \Image (u_{\mathbb{J}_3} (\epsilon)) \] 
\end{proof}

Fixing a pair $(\mathbb{J}, G_{\mathbb{J}})$, and writing $\pi : E_{\mathbb{J}} \rightarrow \mathbb{J}$ for the projection onto the first components, we construct, for every $S \in \mathcal{P} _{Fin} (E_{\mathbb{J}}) $, a new pair $(\mathbb{J}_S , G_{\mathbb{J}_S})$ as follows:
\begin{enumerate}
\item $Ob (\mathbb{J}_S) \coloneqq \underset{T \in \mathcal{P} (S)}{\coprod} Ob (\mathbb{J} / T)$, where $\mathbb{J} / T$ is the multisliced category, defined, for every $T = \{ e_1 , \dots , e_t \}$ with $\pi (e_k) = j_k$, as the category with diagrams in $\mathbb{J}$ of the following shape as objects
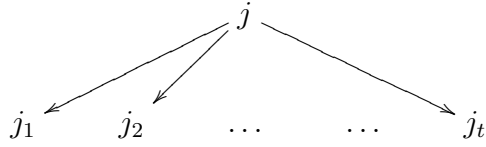
\begin{equation*}
\xymatrix{ & & j\ar[dll]\ar[dl]\ar[drr] & & \\
            j_1 & j_2 & \dots & \dots & j_t}
\end{equation*}
and morphisms defined in the obvious way. If $T = \emptyset$, then $\mathbb{J}/T = \mathbb{J}$. If $T = \{ e_1 \} $ is a singleton, then the category $\mathbb{J} / T$ is the slice category over $ \pi (e_1)$.

\item Given $T = \{ e_1 , \dots , e_t \}, T' = \{ e'_1 , \dots , e'_{t'} \} \subseteq S$, if $T \subseteq T'$, we can define the forgetful functor $H_{T' , T}: \mathbb{J}/T' \rightarrow \mathbb{J}/T$ which acts as 
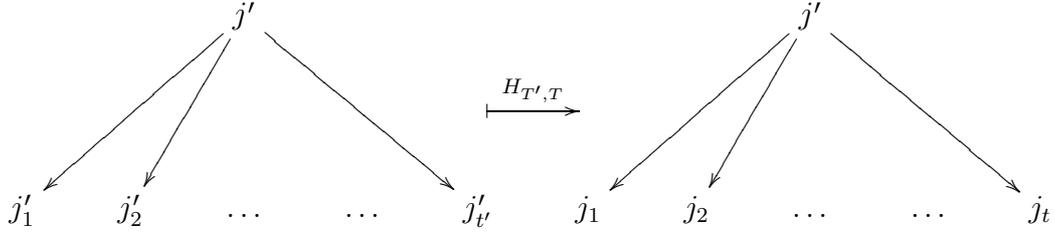
\begin{equation*}
\xymatrix{ & & j'\ar[ddll]\ar[ddl]\ar[ddrr] & & & & & j'\ar[ddll]\ar[ddl]\ar[ddrr] & &\\
            & & & &\ar@{|->}[r]^{H_{T' , T}} & &  \\
            j'_1 & j'_2 & \dots & \dots & j'_{t'} &  j_1 & j_2 & \dots & \dots & j_t}
\end{equation*}
We define $H_{T} \coloneqq H_{T , \emptyset} : \mathbb{J}/T \rightarrow \mathbb{J}$.

The maps of the category $\mathbb{J}_S$ will be generated by:
\begin{itemize}
    \item[$\diamond$] all the maps inside every $\mathbb{J} / T$;
    \item[$\diamond$] for every $T \subseteq T' \subseteq S$, and for every $\hat{j'}\in \mathbb{J} /T'$, a map $ H_{T' , T}^{(\hat{j'})} : \hat{j'} \rightarrow H_{T' , T} (\hat{j'})$. 
\end{itemize}
Precisely, given an object $\hat{j'}\in \mathbb{J} /T'$ and an object $\hat{j} \in \mathbb{J} /T$, then
\begin{equation*}
\Hom_{\mathbb{J}_S} (\hat{j'}, \hat{j}) \cong 
\begin{cases*}
      \Hom_{\mathbb{J} / T} (H_{T' , T} (\hat{j'}) , \hat{j}) & if $T\subseteq T'$ \\
      \emptyset & otherwise
    \end{cases*}
\end{equation*}

\item If $T = \{ e_1 , \dots , e_t\}$ and $\hat{j} \in \mathbb{J} / T$ is the diagram
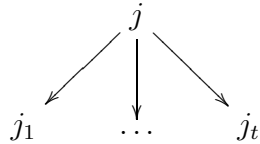
\begin{equation*}
\xymatrix{ & j\ar[dl]\ar[d]\ar[dr] & \\
            j_1 & \dots & j_t}
\end{equation*}
then $G_ {\mathbb{J}_S} (\hat{j})$ is defined as the limit of the diagram

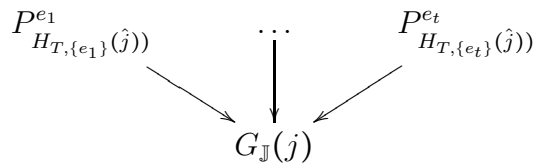
\begin{equation*}
\xymatrix{ P_{H_{T, \{e_1\}}(\hat{j}))}^{e_1}\ar[rd] & \dots\ar[d] & P_{H_{T , \{e_t\}}(\hat{j}))}^{e_t}\ar[ld] \\
& G_{\mathbb{J}} (j) &}
\end{equation*}

where, given $e = (j_1 , f_1, \epsilon_1) \in E_{\mathbb{J}}$ and $\hat{j'} = (j' \rightarrow j_1) \in \mathbb{J} / \{ e \}$, we define $P_{\hat{j'}}^e$ as $ G_{\mathbb{J}}(j')\times _{W} W'$, i.e.\ the pullback
\begin{equation*}
\xymatrix{G_{\mathbb{J}}(j')\times _{W} W'\ar[r]\ar[d] & W'\ar[d]^{\epsilon_1} \\
            G_{\mathbb{J}} (j')\ar[r]_{f_1 \cdot G_{\mathbb{J}}(j' \rightarrow j_1)} & W}    
\end{equation*}

Note that, if $T = \emptyset$, then $G_{\mathbb{J}_S} (\hat{j}) = G_{\mathbb{J}} (\hat{j})$, whereas, if $T = \{e_1 \}$ is a singleton, then $G_{\mathbb{J}_S} (\hat{j}) = P_{\hat{j}}^{e}$.

Given $\hat{j}, \hat{j'} \in \mathbb{J}_S$, the map $G_{\mathbb{J}_S} (\hat{j'} \rightarrow \hat{j})$ is given by the universal property of limits. One can easily check the functoriality of $G_{\mathbb{J}_S}$.
Moreover, as a particular case of this definition, we get, when $\hat{j} = H_{T' , T} (\hat{j'})$, the definition of $G_{\mathbb{J}_S} (H_{T' , T}^{(\hat{j'})})$.

Note that $G_{\mathbb{J}_S}$ does not depend on the order of the elements of $T\subseteq S$.
\end{enumerate}
A few considerations can be made on the data $(\mathbb{J}_S , G_{\mathbb{J}_S})$ .
\begin{Prop}
Every $\mathbb{J}_S$ is the dual of a directed set.
\end{Prop}
\begin{proof}
Define $\overline{J}_S \coloneqq Ob(\mathbb{J}_S)$, and $\leq_S$ the relation given by $\hat{j}\leq_S \hat{j'} \Leftrightarrow (\hat{j'} \rightarrow \hat{j})$. Then $\mathbb{J}_S$ is inhabited (because so is $\mathbb{J}$). Moreover:
\begin{enumerate}
    \item the identity maps in $\mathbb{J}_S$ give the reflexivity of $\leq_S$;
    \item the composition of maps in $\mathbb{J}_S$ gives the transitivity of $\leq_S$; 
    \item for every $\hat{j}, \hat{j'} \in \mathbb{J}_S$, with $\hat{j} \in \mathbb{J}/T$ and $\hat{j'} \in \mathbb{J}/T'$, there exists $z\in\mathbb{J}$ such that $H_T (\hat{j}) \leq_J z$ and $H_{T'} (\hat{j'}) \leq_J z$.
    Called $\hat{z}$ the object
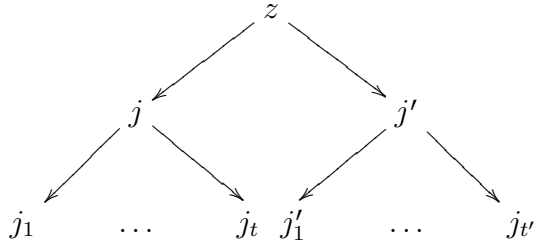
\begin{equation*}
    \xymatrix{ & & z\ar[ld]\ar[rd] & \\
               & j\ar[ld]\ar[rd] & & j'\ar[ld]\ar[rd] & \\
               j_1 & \dots & j_t \ \ j'_1 & \dots & j_{t'} }
\end{equation*}
in $\mathbb{J} / (T \coprod T')$, then in $\mathbb{J}_S$ we have the maps $\hat{z} \rightarrow H_{T \coprod T' , T} (\hat{z})\rightarrow \hat{j}$ and $\hat{z}\rightarrow H_{ T \coprod T' , T'} (\hat{z})\rightarrow \hat{j'}$; hence $\hat{j} \leq_S \hat{z}$ and $\hat{j'}\leq_S \hat{z}$. \qedhere
\end{enumerate}
\end{proof}

\begin{Prop}
Let $\mathcal{E} : (\mathcal{P}_{Fin} (E_{\mathbb{J}}), \subseteq) \rightarrow Sets$ be the diagram which sends every $S$ to $\mathbb{J}_S$ and every inclusion $S \subseteq S'$ to the map $i_{S , S'} : \mathbb{J}_S \rightarrow \mathbb{J}_{S'}$. Then $\mathbb{J}^{+} \coloneqq \underset{S \in \mathcal{P}_{Fin}(E_{\mathbb{J}})}{\colim} \mathcal{E}(S)$ is the opposite of a directed set, and there is a functor $G_{\mathbb{J}^{+}} : \mathbb{J}^{+} \rightarrow \mathcal{C}$, defined as $G_{\mathbb{J}^{+}} (S , \hat{j}) \coloneqq G_{\mathbb{J}_S} (\hat{j})$.
\end{Prop}
\begin{proof}
First of all, note that $\mathcal{P}_{Fin} (E_{\mathbb{J}})$ is a filtered category, since every finite diagram $\mathcal{D} : X \rightarrow \mathcal{P}_{Fin} (E_{\mathbb{J}})$ has a cocone, namely $\underset{S \in \mathcal{D} (X)}{\bigcup} S$. Then $\mathbb{J}^{+} \coloneqq \underset{S \in \mathcal{P}_{Fin}(E_{\mathbb{J}})}{\colim} \mathcal{E}(S)$ is a filtered colimit. It follows that $\mathbb{J}^{+}$ is the opposite of a directed set $(\overline{J^+} , \leq_{\overline{J^+}})$, since every $\mathbb{J}_S$ is so.

$G_{\mathbb{J}^{+}}$ is well defined: if $( S , \hat{j}) = (S' , \hat{j'})$ in $\mathbb{J}^{+}$, then $S \subseteq S'$ and $i_{S , S'} (\hat{j}) = \hat{j'}$ (or $S' \subseteq S$ and $i_{S' , S} (\hat{j'}) = \hat{j}$); hence, $G_{\mathbb{J}^{+}} (S' , \hat{j'}) = G_{\mathbb{J}_{S'}} (\hat{j'}) = G_{\mathbb{J}_{S'}} (i_{S , S'} (\hat{j})) = G_{\mathbb{J}_{S}} (\hat{j}) = G_{\mathbb{J}^{+}} (S , \hat{j})$
(or $G_{\mathbb{J}^{+}} (S , \hat{j}) = G_{\mathbb{J}_{S}} (\hat{j}) = G_{\mathbb{J}_{S}} (i_{S' , S} (\hat{j'})) = G_{\mathbb{J}_{S'}} (\hat{j'}) = G_{\mathbb{J}^{+}} (S' , \hat{j'}) $).
\end{proof}

\begin{Prop}
$(\mathbb{J}^{+} , G_{\mathbb{J}^+})$ is a refinement of every $(\mathbb{J}_S , G_{\mathbb{J}_S})$, and it is good for every element of $E_{\mathbb{J}}$.
\end{Prop}
\begin{proof}
The maps $m_S : \mathbb{J}_S \rightarrow \mathbb{J}^+$ given by the colimit are monomorphisms. Moreover, $G_{\mathbb{J}^+} \cdot m_S = G_{\mathbb{J}_S}$ follows from the way we have defined $G_{\mathbb{J}^+}$.

Consider an element $e = ( j , f, \epsilon) \in E_{\mathbb{J}}$. Because of the way we have defined the generic $G_{\mathbb{J}_S}$, the pair $(\mathbb{J}_{\{e \}} , G_{\mathbb{J}_{\{e\}}})$ is good for $e$. Hence, using Lemma \ref{Lemma_refinements}, we can conclude that the pair $(\mathbb{J}^+ , G_{\mathbb{J}^+})$ is good for $e$ too.
\end{proof}

Neverthless, $(\mathbb{J}^+ , G_{\mathbb{J}^+})$ might not be good for every element of $E_{\mathbb{J}^+}$. In order to obtain a pair $(\mathbb{I} , G_{\mathbb{I}})$ which is good for every element in $E_{\mathbb{I}}$, we make use of the following trick: given $(\mathbb{J} , G_{\mathbb{J}})$, we have constructed
$(\mathbb{J}^+ , G_{\mathbb{J}^+})$. We call $(\mathbb{J}^1, G_{\mathbb{J}^1}) \coloneqq (\mathbb{J}^+ , G_{\mathbb{J}^+})$. Starting from $(\mathbb{J}^1, G_{\mathbb{J}^1})$, we can define a new pair $(\mathbb{J}^2 , G_{\mathbb{J}^2}) \coloneqq ((\mathbb{J}^1)^+, G_{(\mathbb{J}^1)^+})$. By repeating this procedure for every $n\in \mathbb{N}$, we obtain a sequence of embeddings:
\[ \mathbb{J} \coloneqq \mathbb{J}^0 \stackrel{l_{0,1}}{\rightarrow} \mathbb{J}^1 \stackrel{l_{1,2}}{\rightarrow} \mathbb{J}^2 \rightarrow \dots \stackrel{l_{n-1, n}}{\rightarrow} \mathbb{J}^n \rightarrow \dots \]
where, for every $n\in \mathbb{N}$, $\mathbb{J}^{n+1} = (\mathbb{J}^n)^+$.
Moreover, every pair $(\mathbb{J}^{n+1} , G_{\mathbb{J}^{n+1}})$ is a refinement of $(\mathbb{J}^n , G_{\mathbb{J}^n})$, and it is good for every element of $E_{\mathbb{J}^n}$.

\begin{Def}\label{I}
Define $\mathbb{I} \coloneqq \underset{n\in \mathbb{N}}{\colim} \ (\mathbb{J}^n)$ and $G_{\mathbb{I}}: \mathbb{I} \rightarrow \mathcal{C}$ as the functor which sends  $(n , j)$ to $G_{\mathbb{J}^n} (j)$.    
\end{Def}
\begin{Rem}
Note that $G_{\mathbb{I}}$ is well defined, since $(n , j) = (n' , j')$ in $\mathbb{I}$ if and only if $n \leq n'$ (or $n' \leq n$) and $j' = l_{n , n'} (j)$, where $l_{n , n'} \coloneqq l_{n' -1, n'} \cdot \dots \cdot l_{n , n+1}$ (or $j = l_{n' , n} (j')$).
Hence, if $(n , j) = (n' , j')$, then $G_{\mathbb{I}} (n' , j') = G_{\mathbb{J}^{n'}} (j') = G_{\mathbb{J}^{n'}} (l_{n , n'} (j)) = G_{\mathbb{J}^{n}} (j) = G_{\mathbb{I}} (n , j) $ (or $G_{\mathbb{I}} (n , j) = G_{\mathbb{J}^{n}} (j) = G_{\mathbb{J}^{n}} (l_{n' , n} (j')) = G_{\mathbb{J}^{n'}} (j') = G_{\mathbb{I}} (n' , j')$).
Moreover, $(\mathbb{I}, G_{\mathbb{I}})$ is a refinement of each $(\mathbb{J}^n , G_{\mathbb{J}^n})$.
\end{Rem}
\begin{Lemma}
$(\mathbb{I} , G_{\mathbb{I}})$ is good for every $e\in E_{\mathbb{I}}$. 
\end{Lemma}
\begin{proof}
If $e = (\overline{i} , f, \epsilon) \in E_{\mathbb{I}}$, then $\overline{i} = (i,f)$ for some $n\in\mathbb{N}$ and $i\in\mathbb{J}^n$. Moreover, $(\psi_{\mathbb{J}^n , \mathbb{I}})_W (i ,f) = (\overline{i} , f)$. Hence, since $(\mathbb{J}^{n+1} , G_{\mathbb{J}^{n+1}})$ is good for $( i ,f, \epsilon) \in E_{\mathbb{J}^n}$, then, by Lemma \ref{Lemma_refinements}, $(\mathbb{I}, G_{\mathbb{I}})$ is good for it as well. Hence $\psi_{\mathbb{I}, \mathbb{I}}(\overline{i} , f) = (\overline{i} , f) = \psi_{\mathbb{J}^{n} , \mathbb{I}} (i ,f) \in \Image (u_{\mathbb{I}} (\epsilon))$, and so $(\mathbb{I}, G_{\mathbb{I}})$ is good for $e$.
\end{proof}

\subsection{The site \texorpdfstring{$(\mathcal{A} , R)$}{(\textit{A}, R)} has a conservative set of points.}
 Throughout this subsection, once that we have fixed $(\mathbb{J},G_{\mathbb{J}})$, the pair $(\mathbb{I}, G_{\mathbb{I}})$ will be defined as in Definition \ref{I}.

\begin{Lemma}\label{injectivity}
For every $(\mathbb{J}, G_{\mathbb{J}})$, the map $\overline{\psi_{\mathbb{J} , \mathbb{I}}}$ is injective.
\end{Lemma}
\begin{proof}
Consider two elements $(j , t)$ and $(j' , t')$ of $F_{p_{\mathbb{J}}}$. By construction, $\overline{\psi_{\mathbb{J} , \mathbb{I}}}$ sends $(j , t)$ to $(\overline{j} , t)$, where $\overline{j} = (0 , j) \in \mathbb{I}$. In the same way,  $\overline{\psi_{\mathbb{J} , \mathbb{I}}} (j' , t') = (\overline{j'} , t')$, where $\overline{j'} = (0 , j')$.
Now, $(\overline{j} , t) = (\overline{j'} , t')$ means that there exists a map $\overline{j'} \rightarrow \overline{j}$ (or $\overline{j} \rightarrow \overline{j'}$) in $\mathbb{I}$  such that $F(G_{\mathbb{I}} (\overline{j'} \rightarrow \overline{j})) (t) = t'$ (or  $F(G_{\mathbb{I}} (\overline{j} \rightarrow \overline{j'}) (t') = t$).
Note that $\Hom_{\mathbb{I}} (\overline{j'} , \overline{j}) \cong \Hom_{\mathbb{J}} (j' , j)$. Using this, together with the definition of $G_{\mathbb{I}}$, we conclude that there exists $j' \rightarrow j$ (or $j \rightarrow j'$) in $\mathbb{J}$  such that $F(G_{\mathbb{J}} (j' \rightarrow j)) (t) = t'$ (or  $F(G_{\mathbb{J}} (j \rightarrow j')) (t') = t$). Hence $(j , t) = (j' , t')$, and we have the injectivity of $\overline{\psi_{\mathbb{J} , \mathbb{I}}}$.
\end{proof}

To show that $p_{\mathbb{I}}$ is a point, we first recall (part of) a lemma from  \cite[Proposition \stackstag{00YC}]{Stacks-project}; we also sketch of the proof, to witness that the proof is valid in the constructive setting.
\begin{Lemma}\label{prelim}
Let $(\mathcal{C}, I)$ be a site, and suppose that $\mathcal{C}$ is finitely complete. Let $u: \mathcal{C} \rightarrow \Cat{Set}$ be a functor such that
\begin{enumerate}
 \item $u$ commutes with finite limits;
 \item for every $\{ U_i \rightarrow U \}_{i\in I}$ covering family, $\underset{i\in I}{\coprod} u (U_i) \rightarrow u (U)$ is surjective.
\end{enumerate}
Then $u$ is a point. 
\end{Lemma}
\begin{proof}
The only non trivial part of the proof consists in deriving Point (3) of Definition \ref{def_point}, i.e.\ that the stalk functor is left exact. Since finite limits commute with filtered colimits, it is enough to show that the category $\{(U, x) \mid U\in \mathcal{C}, \ x \in u(U)\}$ is cofiltered (so that the opposite of this category turns out to be filtered). Indeed, it is inhabited, since we can take the neighborhood given by the final object in $\mathcal{C}$ together with the element of the singleton set. Then, for every pair of objects $(U, x)$ and $(V, y)$, the fact that $u$ commutes with products gives the existence of a neighborhood $(U \times V , z)$ which can be mapped to both $(U, x)$ and $(V, y)$. Third, it is possible to prove that the category has equalizers, and this completes the proof of the cofilterness.  
\end{proof}

With this we show:
\begin{Lemma}\label{p_I point}
Let $(\mathcal{C}, R)$ be a small site equipped with the regular topology, and suppose $\mathcal{C}$ is finitely complete. Then the functor $u_{\mathbb{I}}$, defined as in \ref{u_J}, gives a point $p_{\mathbb{I}}$.
\end{Lemma}
\begin{proof}
We check that conditions (1) and (2) of Lemma \ref{prelim} are satisfied:
\begin{enumerate}
    \item since it is defined as a filtered colimit in Set, $u_{\mathbb{I}}$ commutes with finite limits;
    \item since covering families are given by single epimorphisms, we have to verify that, for every epimorphism $\epsilon : W' \twoheadrightarrow W$, $u_{\mathbb{I}} (\epsilon) : u_{\mathbb{I}} (W') \rightarrow u_{\mathbb{I}} (W)$ is surjective. This is a consequence of the fact that $(\mathbb{I} , G_{\mathbb{I}})$ is good for $E_{\mathbb{I}}$: indeed, if $(\overline{i} , f)\in u_{\mathbb{I}} (W)$, then $\overline{i} = (n , i)$ for some $n\in \mathbb{N}$ and some $i\in \mathbb{J}^n$. Called $e \coloneqq (i , f, \epsilon)$, then, by construction, $(\mathbb{J}^{n+1}, G_{\mathbb{J}^{n+1}})$ is good for $e$. Since $(\mathbb{I} , G_{\mathbb{I}})$ is a refinement of $(\mathbb{J}^{n+1}, G_{\mathbb{J}^{n+1}})$, then $(\mathbb{I} , G_{\mathbb{I}})$ is good for $e$ too. Hence, $(\overline{i} , f) = \psi_{\mathbb{J}^n , \mathbb{I}}(i , f) \in \Image (u_{\mathbb{I}} (\epsilon))$, and we have the surjectivity of $u_{\mathbb{I}} (\epsilon)$. \qedhere  
\end{enumerate}
\end{proof}

Using Lemma \ref{injectivity} and Lemma \ref{p_I point}, we are able to prove the main result of this section, inspired by the one given in  \cite[Proposition \stackstag{00YQ}]{Stacks-project} for the non constructive treatment.
\begin{Th}\label{main-theorem}
Let $(\mathcal{C} , R)$ be a small site equipped with the regular topology. If $\mathcal{C}$ is finitely complete, then $(\mathcal{C} ,R)$ has a conservative set of points.
\end{Th}
\begin{proof}
To prove the claim, we want to use Lemma \ref{suff_cond}. We start by fixing an object $U$ of $\mathcal{C}$, and by considering the directed set $(\{\ast \} , \leq_J)$, where $\leq_J$ is the trivial relation. Let $\mathbb{J}$ be its dual, and define $G_{\mathbb{J}} (\ast) \coloneqq U$. Constructing the corresponding pair $(\mathbb{I} , G_{\mathbb{I}})$, we obtain a point $p_{\mathbb{I}}$ such that, for any given $F\in \Cat{Sh}(\mathcal{C} , R)$, the map $\overline{\psi_{\mathbb{J} , \mathbb{I}}} : F_{p_{\mathbb{J}}} \rightarrow F_{p_\mathbb{I}}$ is injective. Note that $F_{p_{\mathbb{J}}} = F(U)$ because of the way we have defined $G_{\mathbb{J}}$.

Repeating this for every $U\in \mathcal{C}$, we obtain a colelction of points $\{p_U\}_{U\in \mathcal{C}}$ satisfying the hypothesis of Lemma \ref{suff_cond}.
The collection of all this points is thus a conservative set, and the site $(\mathcal{C} , R)$ has enough points.
\end{proof}

\begin{Rem}
Theorem \ref{main-theorem} might seem surprising, since results on enough points correspond to completeness theorems and are often not provable constructively; however the present case is comparable to the completeness for regular logic proven in \cite[Corollary 4.16]{Forssell-Lumsdaine}, and another proof of the result can be deduced from that.
\end{Rem}

As an  immediate consequence of this, we achieve that:

\begin{Cor}\label{Cor}
If $\mathcal{A}$ is a small abelian category and $R$ is the regular topology on it, then the site $(\mathcal{A} , R)$ has a conservative set of points. 
\end{Cor}
\begin{proof}
The thesis follows from Theorem \ref{main-theorem}, using the finite completeness and the smallness of $\mathcal{A}$. 
\end{proof}

We are now able to state the main goal of the second part of the paper.

\begin{Th}
Any small abelian category $\mathcal{A}$ admits a full exact embedding into the category of sheaves of modules over the ringed space $(X_{\Cat{Sh}(\mathcal{A}, R)}, \mathbb{Z}_{X_{\Cat{Sh} (\mathcal{A}, R)}})$.
\end{Th}
\begin{proof}
This comes as a consequence of the discussion in Section \ref{Description-new-embedding} and of Corollary \ref{Cor}.
\end{proof}

\begin{Rem}
Recalled from \cite[Corollary II.6.3]{MacLane-Moerdijk} that, for every $X$ topological space, $\Cat{Sh} (X)$ is equivalent to $Etale(X)$, and that this last one is a full subcategory of $\Cat{Top} / X$, we can also conclude that any small abelian category $\mathcal{A}$ embeds fully into the category $\Cat{Ab}(\Cat{Top} / X_{\Cat{Sh}(\mathcal{A}, R)})$ of abelian group objects of $\Cat{Top} / X_{\Cat{Sh}(\mathcal{A}, R)}$.
\end{Rem}

\section{Conclusion and further developments}
As promised at the beginning of the paper, we have first shown the constructive issues contained in the proofs of the Freyd--Mitchell Embedding Theorem, and then we presented a constructive way to embed a small abelian category into the category of sheaves of modules over a ringed space. At this point, one might ask if this is the best we can acheive in $\mathbf{CZF}$, or if we can go a step further, and find an embedding into a category of modules over a ring. As far as we know, this question is still open.

\newpage

\appendix

\section{Axioms of \texorpdfstring{$\mathbf{CZF}$}{\textbf{CZF}}}\label{appendix A}
The \textit{Constructive Zermelo--Fraenkel Set Theory} $\mathbf{CZF}$ is formulated in first order intuitionistic logic. Here we briefly list the axioms of the theory. For a more detailed explanation we refer to \cite{CZF}.
\begin{enumerate} \addtolength{\itemsep}{0.2em}
 \item Extensionality
 \[ \forall a \forall b ( \forall x ( x\in a \leftrightarrow x\in b ) \rightarrow a = b ) \]
 
 \item Pairing
\[ \forall a \forall b \exists y \forall x ( x\in y \leftrightarrow ( x = a \vee x = b ) ) \]

\item Union 
\[ \forall a \exists y \forall x ( x\in y \leftrightarrow \exists u \in a (x\in u)) \]

\item Strong Infinity
\[ \exists a ( \Ind{a} \wedge \forall b (\Ind{b} \rightarrow \forall x\in a (x\in b))) \]
where we use the following abbreviations:
\begin{align*}
\Succ{(x,y)} & \coloneqq \forall z(z\in y \leftrightarrow z\in x \vee z = x); \\
    \Ind{a}& \coloneqq (\exists y\in a) (\forall z\in y) \bot \wedge (\forall x\in a) (\exists y\in a) \Succ{(x,y)}.    
\end{align*}

\item Set Induction scheme
\[ \forall a ( \forall x\in a \phi(x) \rightarrow \phi (a) ) \rightarrow \forall a \phi (a) \]
for all formulae $\phi (a)$.

\item Bounded Separation scheme
\[ \forall a \exists y \forall x ( x\in y \leftrightarrow x\in a \wedge \phi (x) ) \]
where $\phi (x)$ is a $\Delta_0$ formula in which $y$ is not free. Here a $\Delta_0$ formula (or ``restricted formula'') is a formula where all the quantifiers are bounded.

\item Strong Collection scheme
\[ \forall x\in a \exists y \phi (x,y) \rightarrow \exists b ( \forall x\in a \exists y\in b \phi(x,y) \wedge \forall y \in b \exists x\in a \phi (x,y)) \]
for every formula $\phi(x,y)$.

\item Subset Collection scheme
\[ \exists c \forall u (\forall x\in a \exists y\in b \psi(x,y,u) \rightarrow \exists d\in c (\forall x\in a \exists y\in d \psi(x,y,u) \wedge \forall y\in d \exists x\in a \psi(x,y,u)))\]
for every formula $\psi(x,y,u)$.
\end{enumerate}

\begin{Rem}
Since we used the Axiom of Exponentiation in Proposition \ref{well-poweredness}, we would like to remark the fact that it is possible to derive the Axiom of Exponentiation from the Subset Collection Scheme. More precisely, in $\mathbf{CZF}$ one can prove that
\begin{center}
Subset Collection Scheme $\Rightarrow$ Axiom of Fullness $\Rightarrow$ Axiom of Exponentiation
\end{center}
where the Axiom of Fullness is 
\begin{center}
    $\forall a \forall b \exists c \Full{c, \mv{a}{b}}$
\end{center}
with
\begin{align*}
 \mv{a}{b} &  \coloneqq \{ r\subseteq a \times b \mid \forall u\in a\exists v\in b ( \langle u,v \rangle \in r) \}; \\
 \Full{c, \mv{a}{b}} & \coloneqq (c\subseteq \mv{a}{b} \wedge (\forall r\in \mv{a}{b}\exists s\in c (s\subseteq r) ) .
 \end{align*}
\end{Rem} 
The two implications are proven in \cite[Theorem 5.1.2]{CZF} in the subsystem $\mathbf{ECST}$ (Elementary Constructive Set Theory) of $\mathbf{CZF}$.

\printbibliography

\end{document}